\documentclass[11pt,dvips,twoside]{article}

\usepackage{pslatex}
\usepackage{fancyhdr}
\usepackage{graphicx}
\usepackage{geometry}

\RequirePackage{amsfonts,amssymb,amsmath,amscd,amsthm}
\RequirePackage{txfonts}
\RequirePackage{graphicx}
\RequirePackage{xcolor}
\RequirePackage{geometry}
\RequirePackage{enumerate}

\def\figurename{Figure} 
\makeatletter
\renewcommand{\fnum@figure}[1]{\figurename~\thefigure.}
\makeatother

\def\tablename{Table} 
\makeatletter
\renewcommand{\fnum@table}[1]{\tablename~\thetable.}
\makeatother

\usepackage{amsmath}
\usepackage{amssymb}
\usepackage{amsfonts}
\usepackage{amsthm,amscd}

\newtheorem{theorem}{Theorem}[section]

\newtheorem{corollary}[theorem]{Corollary}
\newtheorem{proposition}[theorem]{Proposition}

\theoremstyle{example}
\newtheorem{example}[theorem]{Example}

\theoremstyle{definition}
\newtheorem{definition}[theorem]{Definition}

\theoremstyle{remark}
\newtheorem{remark}[theorem]{Remark}

\numberwithin{equation}{section}

\begin{document}

\title{\bfseries\scshape{Classification, Derivations and Centroids of Low-Dimensional  Complex BiHom-Trialgebras}}

\author{\bfseries\scshape Ahmed Zahari\thanks{e-mail address: zaharymaths@gmail.com}\\
Universit\'{e} de Haute Alsace,\\
 IRIMAS-D\'{e}partement de Math\'{e}matiques,\\
 18, rue des Fr\`eres Lumi\`ere F-68093 Mulhouse, France.\\
\bfseries\scshape Bouzid Mosbahi\thanks{e-mail address: mosbahibouzid@yahoo.fr}\\
 University of Sfax, Faculty of Sciences of Sfax, BP 1171, 3000 Sfax, Tunisia.\\
 \bfseries\scshape Imed Basdouri\thanks{e-mail address: basdourimed@yahoo.fr}\\
 University of Gafsa, Faculty of Sciences of Gafsa, 2112 Gafsa, Tunisia.
}

\date{}
\maketitle


\noindent\hrulefill

\noindent {\bf Abstract.}\,The basic objective of this research work is to investigate the stucture of BiHom-associative trialgebras.\,In this respect we build up one important class of BiHom-trialgebras and determine properties of right, left and middle operations in BiHom-associative trialgebras.\,We provide a classification of $n$-dimensional BiHom-trialgebras for $n\leq3$.\,Relying upon the classification result of BiHom-associative trialgebras, the derivations and centroids of low-dimensional BiHom-associative trialgebras are characterized.\,Certain properties of the centroids in light of BiHom-associative trialgebras are reviwed the centroids of low-dimensional BiHom-associative trialgebras are computed.

\noindent \hrulefill

\vspace{.3in}


\vspace{.08in} \noindent \textbf{Keywords}:
 BiHom-associative trialgebra, Classification, Centroid, Derivation, Central derivation.
\vspace{.3in}
\vspace{.2in}

\pagestyle{fancy} \fancyhead{} \fancyhead[EC]{ }
\fancyhead[EL,OR]{\thepage} \fancyhead[OC]{Ahmed Zahari, Bouzid Mosbahi, Imed Basdouri} \fancyfoot{}
\renewcommand\headrulewidth{0.5pt}

\section{Introduction}
A  BiHom-associative trialgebra $(\mathcal{A}, \dashv, \vdash,\bot ,\alpha, \beta)$ involves a vector space, three multiplications and two linear self maps.
It may be regarded as a deformation of an associative algebras, where the associativity condition is twisted by a linear map $\alpha$ and $\beta$, such that when $\alpha=id$ (rep. $\beta=id$), the BiHom-associative trialgebra degenerate to exactly BiHom-triassociative algebras. In the current work, our central focus is upon the structure of BiHom-associative trialgebras. Let $\mathcal{A}$ be an $n$-dimensional
$\mathbb{F}$-linear space and let  $\left\{e_1, e_2, \cdots, e_n\right\}$ be a basis of $\mathcal{A}$. A BiHom-associative trialgebra structure on $\mathcal{A}$ with products $\mu,\, \lambda$ and $\xi$ being
determined by $3n^3$ structure constants $\mathcal{\gamma}_{ij}^k,\, \mathcal{\delta}_{ij}^k$  and $\mathcal{\xi}_{ij}^k$,  are
$e_i\dashv e_j=\sum_{k=1}^n\gamma_{ij}^ke_k,\quad e_i\vdash e_j=\sum_{k=1}^n\delta_{ij}^ke_k$ and  $e_i\bot e_j=\sum_{k=1}^n\xi_{ij}^ke_k$. They are also determined by  $\alpha$ (resp. $\beta$)  which is provided by
${2n^2}$ structure constants $a_{ij}$ (resp. $b_{ij}$), where $\alpha(e_i)=\sum_{j=1}^na_{ji}e_j$
(resp. $\beta(e_j)=\sum_{k=1}^nb_{kj}e_k$). Requiring the algebra structure to be BiHom-associative trialgebra and unital  gives rise to a sub-variety of $\mathcal{BT}as$ of $\mathbb{F}^{3n^3+2n^2}$. Base changes
in $\mathcal{A}$ result in the natural transport of structure action of $GL_n(k)$ on $\mathcal{BT}as$. From this perspective, isomorphism classes of $n$-dimensional BiHom-algebras go in one-to-one correspondence with the orbits of the action of $GL_n(\mathbb{F})$ on $\mathcal{BT}as.$

Classification problems of the BiHom-associative trialgebras using the algebraic  and geometric technique have triggered the widest interest in the derivations and centroids of BiHom-associative trialgebras.
The BiHom-associative trialgebras introduced by Loday \cite{Ld} with a motivation to provide dual dialgebras, have been further investigated with  regard to several areas in mathematics
and physics. The classification of Hom-associative algebras was set forward by \cite{AZAB2} and  A. Zahari and I. Bakayoko who studied the classification of BiHom-associative dialgebras \cite{AI}.

This paper centers around examining the derivations and centroids of finite dimensional associative trialgebras. The algebra of derivations and centroids are highly needed and extremely useful in terms of algebraic and geometric classification problems of algebras.

The current paper is laid out as follows. In the fist section, we introduce the topic alongside with some previously obtained results.
The target of this paper is to introduce and classify derivation as well as centroids of BiHom-associative trialgebras.
In section 2, we explore the structure of BiHom-associative trialgebras.
In section 3, we display the algebraic varieties of BiHom-associative trialgebras, and we exhibit classifications, up to isomorphism, of two-dimensional and three-dimensional
BiHom-associative trialgebras.\\
In Section 4, we  identify the classification of the derivations. Eventually, in Section 5, we  present the classification of the centroids. In this case, the concepts of derivations and  centroids is largely inspired and whetted significant scientific concern from that of finite-dimensional algebras. The algebra of centroids plays an important role in the classification problems as well as in different applications
of algebras.In this study, we elaborated classification results of two and  three-dimensional BiHom-associative trialgebras. All considered algebras and vectors spaces  are supposed
over a field $\mathbb{F}$ of characteristic zero.

\newpage

\section{Structure of BiHom-associative trialgebras}
\begin{definition}\label{tia1}
A BiHom-associative trialgebra is a $6$-tuple $(\mathcal{A}, \dashv, \vdash,\bot ,\alpha, \beta)$ consisting of a  linear space $\mathcal{A}$,  linear maps
 $\dashv, \vdash, \bot : \mathcal{A}\otimes \mathcal{A} \longrightarrow \mathcal{A}$ and  $\alpha,\beta : \mathcal{A}\longrightarrow \mathcal{A}$ satisfying,
for all $x, y, z\in \mathcal{A}$, the following conditions :
\begin{eqnarray}
\alpha\circ\beta&=&\beta\circ\alpha,\label{eq1}\\
(x\dashv y)\dashv\beta(z)&=&\alpha(x)\dashv(y\dashv z)\label{eq6},\\
(x\dashv y)\dashv\beta(z)&=&\alpha(x)\dashv(y\vdash z)=\alpha(x)\dashv(y\bot z),\label{eq3}\\
(x\dashv y)\dashv\beta(z)&=&\alpha(x)\vdash(y\dashv z),\label{eq4}\\
(x\dashv y)\vdash\beta(z)&=&\alpha(x)\vdash(y\vdash z)=(x\bot y)\vdash\beta(z),\label{eq5}\\
(x\vdash y)\vdash\beta(z)&=&\alpha(x)\vdash(y\vdash z)\label{eq2},\\
(x\bot y)\dashv\beta(z)&=&\alpha(x)\bot(y\dashv z),\label{eq7}\\
(x\dashv y)\bot\beta(z)&=&\alpha(x)\bot(y\vdash z),\label{eq8}\\
(x\vdash y)\bot\beta(z)&=&\alpha(x)\vdash(y\bot z)\label{eq9},\\
(x\bot y)\bot\beta(z)&=&\alpha(x)\bot(y\bot z)\label{eq10},
\end{eqnarray}
for all $x, y, z\in \mathcal{A}.$ We call $\alpha$ and $\beta$ ( in this order ) the structure maps of $\mathcal{A}.$
\end{definition}
\begin{remark}\label{rq1}
In addition, $\alpha$ and $\beta$ are endomorphisms with respect to $\dashv, \vdash$ and $\bot$. Then,
 $\mathcal{A}$ is said to be a multiplicative BiHom-associative trialgebra
\begin{eqnarray}
\alpha(x\dashv y)=\alpha(x)\dashv\alpha(y)&;&\beta(x\dashv y)=\beta(x)\dashv\beta(y),\nonumber\\
\alpha(x\vdash y)=\alpha(x)\vdash\alpha(y)&;&\beta(x\vdash y)=\beta(x)\vdash\beta(y),\nonumber\\
\alpha(x\bot y)=\alpha(x)\bot\alpha(y)&;&\beta(x\bot y)=\beta(x)\bot\beta(y),\nonumber
\end{eqnarray}
for all  $x, y \in \mathcal{A}.$
\end{remark}
\begin{definition}
 A morphism of BiHom-associative trialgebra is a linear map
$$\psi : (\mathcal{A}, \dashv, \vdash,\bot, \alpha, \beta)\rightarrow(\mathcal{A}', \dashv',\vdash',\bot', \alpha', \beta')$$ such that
$$\alpha'\circ \psi
=\psi\circ\alpha,\;\; \beta'\circ \psi=\psi\circ\beta$$
and
\begin{eqnarray}
 \psi(x\dashv y)&=&\psi(x)\dashv'\psi(y),\nonumber\\
\psi(x\vdash y)&=&\psi(x)\vdash'\psi(y),\nonumber\\
\psi(x\bot y)&=&\psi(x)\bot'\psi(y),\nonumber
\end{eqnarray}
 for all  $x, y \in \mathcal{A}.$
\end{definition}

\begin{remark}
A bijective homomorphism is an isomorphism of $\mathcal{A}_1$ and $\mathcal{A}_2$.
\end{remark}

\begin{definition}
Let $(\mathcal{A}, \ast, \alpha, \beta)$ be a BiHom-associative trialgebra. If there is an associative trialgebra $(\mathcal{A}, \ast')$ such that
$x\ast'y=\alpha^{-1}(x)\ast\beta^{-1}(y),\, \forall x,y\in A$, we say that $(\mathcal{A}, \ast, \alpha,\beta)$ is of an associative type and
$({A},\ast')$ is its compatible associative algebra or the untwist of $(\mathcal{A}, \ast, \alpha, \beta)$.
\end{definition}

\begin{proposition}\label{p1}
Let $(\mathcal{A}, \dashv, \vdash,\bot ,\alpha, \beta)$ be an $n$-dimensional BiHom-associative trialgebra and let $\psi : \mathcal{A}\rightarrow \mathcal{A}$ be an invertible linear map.
Therefore, there is an isomorphism with an n-dimensional BiHom-associative trialgebra  $(\mathcal{A}, \dashv', \vdash', \bot', \psi\alpha\psi^{-1}, \psi\beta\psi^{-1})$, where
$(\dashv', \vdash', \bot')=\psi\circ(\dashv, \vdash, \bot)\circ(\psi^{-1}\otimes\psi^{-1}$). Furthermore, if $\left\{\phi^k_{ij}, \gamma^k_{ij}, \delta^k_{ij} \right\}$ are
the structure constants of $(\dashv, \vdash, \bot)$ with respect to the basis $\left\{e_1,\dots,e_n\right\}$,  then $(\dashv', \vdash', \bot')$ has the same structure constants
 with respect to the basis $\left\{\psi(e_1),\dots,\psi(e_n)\right\}$.
\end{proposition}

\begin{proof}
We prove, for any invertible linear map $\psi : \mathcal{A}\rightarrow \mathcal{A}, that \, (\mathcal{A}, \dashv', \vdash',  \bot',  \psi\alpha\psi^{-1}, \psi\beta\psi^{-1})$ is a
BiHom-associative trialgebra.
$$\begin{array}{ll}
(x\dashv'y)\dashv'\psi\beta\psi^{-1}(z)
&=\psi\dashv(\psi^{-1}\otimes\psi^{-1})(x,y)\dashv'\psi\beta\psi^{-1}(z)\\
&=(\psi(\psi^{-1}\otimes\psi^{-1})(x\dashv y))\dashv'\psi\beta\phi^{-1}(z)\\
&=\psi(\psi^{-1}\otimes\psi^{-1})(\psi(\psi^{-1}\otimes\psi^{-1})(x\dashv y)\dashv\phi^{-1}\beta\psi(z)\\
&=\psi(\psi^{-1}(x)\dashv\psi^{-1}(y))\dashv\psi^{-1}\beta\psi(z)\\
&=\psi(\alpha\psi^{-1}(x)\dashv(\psi^{-1}(y)\vdash\psi^{-1}(z)))\\
&=\psi(\psi^{-1}\otimes\psi^{-1})(\psi\otimes\psi)\alpha\psi^{-1}(x)\dashv(\psi^{-1}\otimes\psi^{-1}(y\vdash z))\\
&=\psi(\psi^{-1}\otimes\psi^{-1})(\psi\alpha\psi^{-1}(x)\dashv(\psi(\psi^{-1}\otimes\psi^{-1})(y\vdash z)))\\
&=\psi\alpha\psi^{-1}(x)\dashv'(y\vdash'z).
\end{array}$$
Hence, $(\mathcal{A}, \dashv', \vdash', \bot',  \psi\alpha\psi^{-1}, \psi\beta\psi^{-1})$ is a BiHom-associative trialgebra.\\
It is also multiplicative. Indeed, for $\alpha$
$$\begin{array}{ll}
\psi\alpha\psi^{-1}(x\dashv'y)
&=\psi\alpha\psi^{-1}\psi\dashv(\psi^{-1}\otimes\psi^{-1})(x,y)\\
&=\psi\alpha\dashv(\psi^{-1}\otimes\psi^{-1})(x,y)\\
&=\psi(\alpha\psi^{-1}(x)\dashv\alpha\psi^{-1}(y))\\
&=\psi((\psi^{-1}\otimes\psi^{-1})(\psi\otimes\phi)\alpha\psi^{-1}(x)\dashv\alpha\psi^{-1}(y))
=(\psi\alpha\psi^{-1}(x)\dashv'\psi\alpha\psi^{-1}(y)).
\end{array}$$

We have also for $\beta$
$$\begin{array}{ll}
\psi\beta\psi^{-1}(x\dashv'y)
&=\psi\beta\psi^{-1}\psi\dashv(\psi^{-1}\otimes\psi^{-1})(x,y)\\
&=\psi\beta\dashv(\psi^{-1}\otimes\psi^{-1})(x,y)\\
&=\psi(\beta\psi^{-1}(x)\dashv\beta\psi^{-1}(y))\\
&=\psi((\psi^{-1}\otimes\psi^{-1})(\psi\otimes\phi)\beta\psi^{-1}(x)\dashv\beta\psi^{-1}(y))
=(\psi\beta\psi^{-1}(x)\dashv'\psi\beta\psi^{-1}(y)).
\end{array}$$
Thus, $\psi : (\mathcal{A}, \dashv, \vdash, \bot,  \alpha, \beta)\rightarrow(\mathcal{A}, \dashv', \vdash', \bot',  \psi\alpha\psi^{-1}, \psi\beta\psi^{-1})$ is a
BiHom-associative trialgebra morphism, since\\  $\psi\circ\dashv=\psi\circ\dashv\circ(\psi^{-1}\otimes\psi^{-1})\circ(\psi\otimes\psi)=\mu'\circ(\psi\otimes\psi)$ and
$(\psi\alpha\phi^{-1})\circ\psi=\psi\circ\alpha$ and $(\psi\beta\psi^{-1})\circ\psi=\psi\circ\beta.$ \\ It is easy to infer that
$\left\{\psi(e_i), \cdots, \psi(e_n)\right\}$ is a basis of $\mathcal{A}$. For $i,j=1,\cdots,n$, we have

$\begin{array}{ll}
(\psi(e_i)\dashv\psi(e_j))
&=\psi(\psi^{-1}(e_i)\dashv \psi^{-1}(e_j))=\psi(e_i\dashv e_j)=\sum_{k=1}^n\gamma^k_{ij}\psi(e_k).
\end{array}$
\end{proof}

\begin{proposition}
Let $(\mathcal{A},  \dashv, \vdash, \bot, \alpha, \beta)$ be a BiHom-associative trialgebra over $\mathbb{F}$.\\ Let $(\mathcal{A}, \dashv', \vdash', \bot', \phi\alpha\psi^{-1}, \psi\beta\psi^{-1})$ be
its isomorphic BiHom-associative trialgebra described in Proposition \ref{p1}. If $\psi$ is an automorphism of $(\mathcal{A}, \dashv, \vdash, \bot,  \alpha, \beta)$, then
$\psi\phi\psi^{-1}$ is an automorphism of\\ $(\mathcal{A}, \dashv',\vdash', \bot',  \psi\alpha\psi^{-1}, \psi\beta\psi^{-1})$.
\end{proposition}

\begin{proof}
Note that $\gamma=\psi\alpha\psi^{-1}$. It follows that
$$\psi\phi\psi^{-1}\gamma=\psi\phi\psi^{-1}\psi\alpha\psi^{-1}=\psi\phi\alpha\psi^{-1}=\psi\alpha\phi\psi^{-1}=\psi\alpha\psi^{-1}\psi\phi\psi^{-1}=\gamma\psi\phi\psi^{-1}.$$
For $\beta$, we assume $\lambda=\psi\beta\psi^{-1}$. As a result, we get\\
$$\psi\phi\psi^{-1}\lambda=\psi\phi\psi^{-1}\psi\beta\psi^{-1}=\psi\phi\beta\psi^{-1}=\psi\beta\phi\psi^{-1}=\psi\beta\psi^{-1}\psi\phi\psi^{-1}=\lambda\psi\phi\psi^{-1}.$$
For any $x,y\in T,$
$$\begin{array}{ll}
 \psi\phi\psi^{-1}(\psi(x)\dashv'\psi(y))&=\psi\phi\psi^{-1}\psi(x\dashv y)
 =\psi\phi(x\dashv y)=\psi(\phi(x)\dashv\phi(y))\\
 &=(\psi\phi(x)\dashv'\psi\phi(y))
 =(\psi\phi\psi^{-1}(\psi(x)\dashv'\phi\psi\phi^{-1}(\phi(y))).
 \end{array}$$
 By Definition, $\psi\phi\psi^{-1}$ is an automorphism of $(A, \dashv', \vdash',  \bot', \psi\alpha\psi^{-1}, \psi\beta\psi^{-1})$.
 \end{proof}


In the below proposition, we prove that BiHom-associative trialgebras are closed under a direct sum, and provide a condition for which
a linear map becomes a morphism of direct sum of BiHom-associative trialgebras.

\begin{proposition}
Let $(\mathcal{A},  \dashv_\mathcal{A}, \vdash_\mathcal{A}, \bot_\mathcal{A}, \alpha_\mathcal{A}, \beta_\mathcal{A})$ and
$(\mathcal{B}, \dashv_\mathcal{B}, \vdash_\mathcal{B},\bot_\mathcal{B},  \alpha_\mathcal{B}, \beta_\mathcal{A})$ be two
 BiHom-associative trialgebras. Then, there exists a BiHom-associative trialgebra structure
on $\mathcal{A}\oplus\mathcal{B}$ with the bilinear maps $\triangleleft, \triangleright, \ast : (\mathcal{A}\oplus\mathcal{B})^{\otimes 2}\rightarrow\mathcal{A}\oplus\mathcal{B}$
  indicated by
$$(a_1+b_2)\dashv(a_2+b_2) :=a_1\dashv_\mathcal{A}a_2+b_2\dashv_\mathcal{B}b_2,$$
$$(a_1+b_1)\vdash(a_2+b_2) :=a_1\vdash_\mathcal{A}a_2+b_\mathcal{A}\vdash_\mathcal{B}b_2$$
$$(a_1+b_1)\bot(a_2+b_2) :=a_1\bot_\mathcal{A}a_2+b_\mathcal{A}\bot_\mathcal{B}b_2$$
and the linear
 maps $\alpha=\alpha_\mathcal{A}+\alpha_\mathcal{B},\, \beta=\beta_\mathcal{A}+\beta_\mathcal{B} : \mathcal{A}\oplus\mathcal{B}\rightarrow\mathcal{A}\oplus\mathcal{B}$ expressed by
$$(\alpha_\mathcal{A}+\alpha_\mathcal{B})(a+b) :=\alpha_\mathcal{A}(a)+\alpha_\mathcal{B}(b),\, (\beta_\mathcal{A}+\beta_\mathcal{B})(a+b) :=\beta_\mathcal{A}(a)+\beta_\mathcal{B}(b),\,
\forall(a,b)\in \mathcal{A}\times\mathcal{B}.
$$
Moreover, if $\xi : \mathcal{A}\rightarrow\mathcal{B}$  is a linear map,
then,
$$ \xi : (\mathcal{A}, \dashv_\mathcal{A}, \vdash_\mathcal{A}, \alpha_\mathcal{A}, \beta_\mathcal{A})
\rightarrow(\mathcal{B}, \dashv_\mathcal{B}, \vdash_\mathcal{B},  \alpha_\mathcal{B}, \beta_\mathcal{B})$$
 is a morphism if and only if its graph  $\Gamma_\xi=\{(x, \xi(x)), x\in \mathcal{A}\}$  is a BiHom-subalgebra of $(\mathcal{A}\oplus\mathcal{B}, \triangleleft, \triangleright,\ast ,\alpha, \beta)$.
\end{proposition}

\begin{proof}
The proof of the first part derives from a direct computation. For this reason, we omit it. Now,
let us suppose that
$\xi : (\mathcal{A}, \dashv_\mathcal{A}, \vdash_\mathcal{A},\alpha_\mathcal{A}, \beta_\mathcal{A})\rightarrow(\mathcal{B}, \dashv_\mathcal{B}, \vdash_\mathcal{B},\alpha_\mathcal{B}, \beta_\mathcal{B})$
is a morphism of BiHom-associative trialgebras.
Consequently $$(u+\xi(u))\dashv(v+\xi(v))=(u\dashv_\mathcal{A}v+\xi(u)\dashv_{B}\xi(v))=u\dashv_\mathcal{A}v+\xi(u\dashv_\mathcal{A}v),$$
     $$(u+\xi(u))\vdash(v+\xi(v))=(u\vdash_\mathcal{A}v+\xi(u)\vdash_{B}\xi(v))=u\vdash_\mathcal{A}v+\xi(u\vdash_\mathcal{A}v).$$
		 $$(u+\xi(u))\bot(v+\xi(v))=(u\bot_\mathcal{A}v+\xi(u)\vdash_{B}\xi(v))=u\bot_{A}v+\xi(u\bot_\mathcal{A}v).$$
Thus, the graph $\Gamma_\xi$ is closed under the operations $\dashv, \vdash$ and $\bot$.\\
 Additionally, since $\xi\circ\alpha_\mathcal{A}=\alpha_\mathcal{B}\circ\xi,$ and $\xi\circ\beta_\mathcal{A}=\beta_\mathcal{B}\circ\xi,$ we have
$$
(\alpha_\mathcal{A}\oplus\alpha_\mathcal{B})(u, \xi(u))=(\alpha_\mathcal{A}(u), \alpha_\mathcal{B}\circ\xi(u))=(\alpha_\mathcal{A}(u), \xi\circ\alpha_\mathcal{A}(u)),
$$
and
$$
(\beta_\mathcal{A}\oplus\beta_\mathcal{B})(u, \xi(u))=(\beta_\mathcal{A}(u), \beta_\mathcal{B}\circ\xi(u))=(\beta_\mathcal{A}(u), \xi\circ\beta_\mathcal{A}(u));
$$
which implies that  $\Gamma_\xi$ is closed $\alpha_\mathcal{A}\oplus\alpha_\mathcal{B}$ and $\beta_\mathcal{A}\oplus\beta_\mathcal{B}.$
 Thus, $\Gamma_\xi$  is a BiHom-subalgebra of $(\mathcal{A}\oplus\mathcal{B}, \dashv, \vdash, \alpha, \beta).$ \\
Conversely, if the graph $\Gamma_\xi\subset\mathcal{A}\oplus\mathcal{B}$ is a BiHom-subalgebra of
$(\mathcal{A}\oplus\mathcal{B}, \dashv, \vdash, \alpha, \beta)$, then we get
$$(u+\xi(u))\dashv(v+\xi(v))=(u\dashv_\mathcal{A}v+\xi(u)\dashv_\mathcal{B}\xi(v))\in \Gamma_\xi, $$
 $$(u+\xi(u))\vdash(v+\xi(v))=(u\vdash_\mathcal{A}v+\xi(u)\vdash_\mathcal{B}\xi(v))\in \Gamma_\xi,$$
$$(u+\xi(u))\bot(v+\xi(v))=(u\bot_\mathcal{A}v+\xi(u)\bot_\mathcal{B}\xi(v))\in \Gamma_\xi.$$
Furthermore, $(\alpha_\mathcal{A}\oplus\alpha_\mathcal{B})(\Gamma_\xi)\subset \Gamma_\xi,\, and (\beta_\mathcal{A}\oplus\beta_\mathcal{B})(\Gamma_\xi)\subset \Gamma_\xi,$ imply that
$$
(\alpha_\mathcal{A}\oplus\alpha_\mathcal{B})(u, \xi(u))=(\alpha_\mathcal{A}(u),\alpha_\mathcal{B}\circ\xi(u))\in \Gamma_\xi,\,(\beta_\mathcal{A}\oplus\beta_\mathcal{B})(u, \xi(u))
=(\beta_\mathcal{A}(u),\beta_B\circ\xi(u))\in \Gamma_\xi,
$$
which is equivalent to the condition $\alpha_\mathcal{B}\circ\xi(u)=\xi\circ\alpha_\mathcal{A}(u),$ i.e., $\alpha_\mathcal{B}\circ\xi=\xi\circ\alpha_\mathcal{A}.$\\
Similarly, $\beta_\mathcal{B}\circ\xi=\xi\circ\beta_\mathcal{A}$. Therefore, $\xi$ is a morphism of BiHom-associative trialgebras.
\end{proof}

\begin{definition}
Let $(\mathcal{A}, \dashv, \vdash, \bot,  \alpha, \beta)$ be a BiHom-associative trialgebra endowed with a linear  map $R : \mathcal{A}\longrightarrow  \mathcal{A}$  such that
$\alpha\circ R=R\circ\alpha$, $\beta\circ R=R\circ\beta$ and
$$\begin{array}{ll}
R(x)\vdash R(y)&=R(R(x)\dashv y+x\dashv R(y)+\lambda(x\dashv y)),
 \end{array}$$
$$\begin{array}{ll}
R(x)\dashv R(y)&=R(R(x)\vdash y+x\vdash R(y)+\lambda(x\vdash y)),
 \end{array}$$
$$\begin{array}{ll}
R(x)\bot R(y)&=R(R(x)\bot y+x\bot R(y)+\lambda(x\bot y)),
 \end{array}$$
with $\lambda\in \mathbb{K}, x, y\in \mathcal{A}$, is called a Rota-Baxter BiHom-associative trialgebra, and R is called a Rota-Baxter operator of weight $\lambda$ on $\mathcal{A}$.
\end{definition}

\begin{example}
We consider the $2$-dimensional BiHom-associative trialgebra with a basis $\left\{e_1, e_2\right\}$.
For $e_1\dashv e_2=e_1,\, e_2\dashv e_1=e_1,\,e_1\vdash e_2=e_1,\, e_1\bot e_2=e_1,e_2\bot e_2=e_1,\, \alpha(e_2)=e_1,\, \beta(e_2)=e_1$ and
 $R(e_1)=-\lambda e_1,\, R(e_2)=-\lambda e_2.$
\end{example}

\begin{proposition}
Let $(\mathcal{A}, \ast, \alpha, \beta, R)$ be some Rota-Baxter BiHom-associative trialgebras of weight $\lambda$, and three new operations $\bot, \dashv$  and $\vdash$ on $\mathcal{A}$ are defined by
$x\dashv y=x\ast R(y),\,x\vdash y= R(x)\ast y$ and $x\bot y=\lambda x\ast y$. Then, $(\mathcal{A},\dashv,\vdash, \bot, \alpha, \beta)$ is a BiHom-associative trialgebra.
\end{proposition}
\begin{proof}
We prove only one axiom, as others are proven similarly. Indeed, for any $x, y, z\in\mathcal{A},$
$$\begin{array}{ll}
(x\vdash y)\dashv\beta(z)&=(R(x)\ast y)\dashv\beta(z)\\
&=(R(x)\ast y)\ast R(\beta(z))\\
&=\alpha(R(x))\ast(y\dashv y)\\
&=\alpha(x)\vdash(y\dashv y),
 \end{array}$$
which ends the proof.
\end{proof}

\begin{proposition}
Let $(\mathcal{A}, \dashv,\vdash, \bot, \alpha, \beta)$ be a BiHom-associative trialgebra and suppose $\alpha^2=\beta^2\alpha\circ\beta=\beta\circ\alpha=id.$
Then,  $(A,\dashv,\vdash, \bot, \alpha, \beta)\cong (\mathcal{A}, \dashv,\vdash, \bot, \beta,\alpha).$
\end{proposition}
\begin{proof}
We prove only one axiom, as others are proven similarly. For any $x, y, z\in\mathcal{A},$
$$\begin{array}{ll}
\alpha(x)\dashv (y\vdash z)&=\alpha(x)\dashv(y\bot z)\\
&\Leftrightarrow\alpha(\alpha\beta(x))\dashv (y\vdash z)=\alpha(\alpha\beta(x))\dashv(y\bot z)\\
&\Leftrightarrow\alpha^2\beta(x)\dashv (y\vdash z)=\alpha^2\beta(x)\dashv(y\bot z)\\
&\Leftrightarrow\beta(x)\dashv (y\vdash z)=\beta(x)\dashv(y\bot z).
 \end{array}$$
Hence,  $(\mathcal{A},\bot, \dashv,\vdash,\alpha, \beta)\cong (\mathcal{A},\bot, \dashv,\vdash, \beta,\alpha).$
\end{proof}

\begin{proposition}
Let $(\mathcal{A}, \dashv,\vdash, \bot, \alpha, \beta)$ be a BiHom-associative trialgebra. Therefore, $(\mathcal{A}, \dashv, \bot, \ast, \alpha, \beta)$ is a BiHom-associative trialgebra, where for any
$x, y, z\in \mathcal{A},$ $x\ast y=x\vdash y+x\bot y.$
\end{proposition}
\begin{proof}
We prove only one axiom, as others are proven similarly. For any $x, y, z\in \mathcal{A},$
$$\begin{array}{ll}
(x\ast y)\dashv \beta(z)&=(x\vdash y+x\bot y)\dashv \beta(z)\\
&=(x\vdash y)\dashv \beta(z)+(x\bot y)\dashv \beta(z)\\
&=\alpha(x)\vdash(y\dashv z)+\alpha(x)\bot(y\dashv z)=\alpha(x)\ast(y\dashv z),
 \end{array}$$
which ends the proof.
\end{proof}

\begin{proposition}
Let $(\mathcal{A},\dashv,\vdash, \bot, \alpha, \beta)$ be BiHom-associative trialgebra. Define now binany operationqs by
$$\begin{array}{ll}
x\ast y&=x\dashv y-y\vdash x\\
\left[x, y\right]&=x\bot y-y\bot x.
 \end{array}$$
Then $(\mathcal{A}, \ast, \left[, \right], \alpha, \beta)$ becames a BiHom-associative trialgebra.
\end{proposition}

\begin{proof}
By define, for any $x, y, z\in \mathcal{A}$ we have
$$\begin{array}{ll}
\left[x, y\right]\ast\beta(z)&=\left[x, y\right]\dashv\beta(z)-\alpha(z)\vdash\left[x, y\right]\\
&=(x\vdash y-y\bot x)\dashv \beta(z)-\alpha(z)\vdash(x\bot y-y\bot y).
 \end{array}$$
and
$$\begin{array}{ll}
\left[x\ast z,\beta(y)\right]+\left[\alpha(x), y\ast z\right]&=\left[x\dashv z-z\vdash, \beta(y)\right]+\left[\alpha(x), y\dashv z-z\vdash y\right]\\
&=(x\vdash z-z\vdash x)\bot\beta(y)-\alpha(y)\bot(x\dashv z-z\vdash x)\\
&+\alpha(x)\bot(y\dashv z-z\vdash y)-(y\dashv z-z\vdash y)\bot\beta(x).\\
 \end{array}$$
Since
$$\begin{array}{ll}
(x\bot y)\dashv\beta(z)&=\alpha(x)\bot(y\vdash z),\\
(y\bot x)\dashv\beta(z)&=\alpha(y)\bot(x\vdash z),\\
\alpha(z)\vdash (x\bot y)&=(z\vdash x)\bot\beta(y),\\
\alpha(z)\vdash (y\bot x)&=(z\vdash y)\bot\beta(x),\\
(x\bot z)\bot\beta(y)&=\alpha(x)\bot(z\vdash y),\\
\alpha(y)\bot(z\vdash x)&=(y\dashv z)\bot\beta(x),\\
 \end{array}$$
in a  BiHom-associative trialgebra, we have
$\left[x, y\right]\ast\alpha\beta(z)=\left[x\ast z, \beta(y)\right]+\left[\alpha(x), y\ast z\right].$
This ends the proof.
\end{proof}
\begin{proposition}
Let $(\mathcal{A}, \dashv,\vdash, \bot, \alpha, \beta)$ be a BiHom-associative trialgebra and  $x\ast y=x\vdash y+x\dashv y+x\bot y.$ Then, $(\mathcal{A},\ast,\alpha, \beta)$ is a
BiHom-associative algebra.

\end{proposition}
\begin{proof}
For any $x, y, z\in \mathcal{A},$
$$\begin{array}{ll}
(x\ast y)\ast \beta(z)&-\alpha(x)\ast(y\ast z)\\
&=(x\vdash y+x\dashv y+x\bot y)\ast\beta(z)-\alpha(x)\ast(y\vdash z+y\dashv z+y\bot z)\\
&=(x\vdash y)\ast\beta(z)+(x\dashv y)\ast\beta(z)+(x\bot y)\ast\beta(z)-\alpha(x)\ast(y\vdash z)-\alpha(x)\ast(y\dashv z)-\alpha(x)\ast(y\bot z)\\
&=(x\vdash y)\vdash\beta(z)+(x\vdash y)\dashv\beta(z)+(x\vdash y)\bot\beta(z)+(x\dashv y)\vdash\beta(z)+(x\dashv y)\dashv\beta(z)+(x\dashv y)\bot\beta(z)\\
&+(x\bot y)\vdash\beta(z)+(x\bot y)\dashv\beta(z)+(x\bot y)\bot\beta(z)-\alpha(x)\vdash(y\vdash z)-\alpha(x)\dashv(y\vdash z)-\alpha(x)\bot(y\vdash z)\\
&-\alpha(x)\vdash(y\dashv z)-\alpha(x)\dashv(y\dashv z)-\alpha(x)\bot(y\dashv z)-\alpha(x)\vdash(y\bot z)-\alpha(x)\dashv(y\bot z)-\alpha(x)\bot(y\bot z).
 \end{array}$$
This confirms that $(\mathcal{A},\ast,\alpha, \beta)$ is a BiHom-associative algebra.
\end{proof}

\begin{definition}
An averaging operator over a BiHom-associative trialgebra $(\mathcal{A}, \dashv, \bot, \ast,\alpha, \beta)$ is a linear map $\xi : \mathcal{A}\longrightarrow \mathcal{A}$ such that
$\alpha\circ\xi=\xi\circ\alpha$ and $\beta\circ\xi=\xi\circ\beta$, and for all $x, y\in\mathcal{A}, $
$$\begin{array}{ll}
\xi(\xi(x)\bot y)&=\xi(x)\bot\xi(y)=\xi(x\bot \xi(y)),
 \end{array}$$
$$\begin{array}{ll}
\xi(\xi(x)\dashv y)&=\xi(x)\dashv\xi(y)=\xi(x\dashv \xi(y)),
 \end{array}$$
$$\begin{array}{ll}
\xi(\xi(x)\vdash y)&=\xi(x)\vdash\xi(y)=\xi(x\vdash \xi(y)).
 \end{array}$$
\end{definition}

\begin{proposition}
Let  $(\mathcal{A}, \cdot)$ be an associative algebra and let $\alpha, \beta : \mathcal{A}\longrightarrow \mathcal{A}$ be two averaging operators such that $(\mathcal{A}, \cdot, \alpha, \beta)$
is a BiHom-associative algebra. For any $x, y\in \mathcal{A}, $ define new operators on $\mathcal{A}$ by
$x\dashv y=\alpha(x)\cdot y,\quad x\vdash y=x\cdot\beta(y),\quad x\bot y=\alpha(x)\cdot\beta(y).$
\end{proposition}

\begin{proof}
We prove only one axiom, as others are proven similarly. For any $x, y, z\in\mathcal{A},$
$$\begin{array}{ll}
(x\dashv y)\bot\beta(z)&=(\alpha(x)\cdot y)\bot\beta(z)\\
&=\alpha(\alpha(x)\cdot y)\cdot\beta(\beta(y))\\
&=(\alpha(x)\cdot\alpha(y))\cdot\beta(\beta(z))=\alpha(\alpha(x))\cdot(\alpha(y)\cdot\beta(z)),
 \end{array}$$
which ends the proof.
\end{proof}

\section{Classification of low-dimensional BiHom-associative trialgebras}
This section provides a detailed description of $\alpha\beta$-derivation of BiHom-associative trialgebras in dimension two and three over the field $\mathbb{F}.$ Let
$\left\{e_1,e_2, e_3,\cdots, e_n\right\}$ be a basis of an n-dimensional BiHom-associative trialgebra $\mathcal{A}.$ The product of basis is expressed in terms of
$$e_i\dashv e_j=\sum_{k=1}^n\gamma_{ij}^ke_k; \quad e_i\vdash e_j=\sum_{k=1}^n\delta_{ij}^ke_k; \quad e_i\bot e_j=\sum_{k=1}^n \xi_{ij}^ke_k;\,\,
\alpha(e_i)=\sum_{j=1}^na_{ji}e_j; \quad \beta(e_j)=\sum_{k=1}^nb_{kj}e_k.$$
We have
\begin{eqnarray}
\sum_{j=1}^nb_{ji}a_{kj}&=&\sum_{j=1}^na_{ji}b_{kj},\\
\sum_{p=1}^n\sum_{q=1}^n\gamma_{ij}^pb_{qk}\gamma_{pq}^r&=&\sum_{p=1}^n\sum_{q=1}^na_{pi}\gamma_{jk}^q\gamma_{pq}^r,\\
\sum_{p=1}^n\sum_{q=1}^n\gamma_{ij}^pb_{qk}\gamma_{pq}^r&=&\sum_{p=1}^n\sum_{q=1}^na_{pi}\delta_{jk}^q\gamma_{pq}^r=\sum_{p=1}^n\sum_{q=1}^na_{pi}\xi_{jk}^q\gamma_{pq}^r,\\
\sum_{p=1}^n\sum_{q=1}^n\delta_{ij}^pb_{qk}\gamma_{pq}^r&=&\sum_{p=1}^n\sum_{q=1}^na_{pi}\gamma_{jk}^q\delta_{pq}^r,\\
\sum_{p=1}^n\sum_{q=1}^n\gamma_{ij}^pb_{qk}\delta_{pq}^r&=&\sum_{p=1}^n\sum_{q=1}^na_{pi}\gamma_{jk}^q\delta_{pq}^r=\sum_{p=1}^n\sum_{q=1}^n\xi_{ij}^pb_{qk}\delta_{pq}^r,\\
\sum_{p=1}^n\sum_{q=1}^n\delta_{ij}^pb_{qk}\delta_{pq}^r=\sum_{p=1}^n\sum_{q=1}^na_{pi}\delta_{jk}^q\delta_{pq}^r &;&\quad
\sum_{p=1}^n\sum_{q=1}^n\xi_{ij}^pb_{qk}\gamma_{pq}^r=\sum_{p=1}^n\sum_{q=1}^na_{pi}\gamma_{jk}^q\xi_{pq}^r,\\
\sum_{p=1}^n\sum_{q=1}^n\gamma_{ij}^pb_{qk}\xi_{pq}^r=\sum_{p=1}^n\sum_{q=1}^na_{pi}\delta_{jk}^q\xi_{pq}^r&;&
\sum_{p=1}^n\sum_{q=1}^n\delta_{ij}^pb_{qk}\xi_{pq}^r=\sum_{p=1}^n\sum_{q=1}^na_{pi}\xi_{jk}^q\delta_{pq}^r,\\
\sum_{p=1}^n\sum_{q=1}^n\xi_{ij}^pb_{qk}\xi_{pq}^r&=&\sum_{p=1}^n\sum_{q=1}^na_{pi}\xi_{jk}^q\xi_{pq}^r,
\end{eqnarray}
where ${(a_{ij})}_{1\leq i, j\leq n}$ refers to the matrix of $\alpha$, ${(b_{ij})}_{1\leq i, j\leq n}$ corresponds to the matrix of $\beta$ and $(\gamma_{ij}^k)$, $(\delta_{ij}^k)$ and $(\xi_{ij}^k)$
stand for the structure constants of $\mathcal{A}.$

The axioms in Remark \ref{rq1} are,respectively,equivalent to
\begin{eqnarray}
\sum_{k=1}^n\gamma_{ij}^ka_{qk}=\sum_{k=1}^n\sum_{p=1}^na_{ki}a_{pj}\gamma_{kp}^q,\quad \sum_{k=1}^n\delta_{ij}^ka_{qk}=\sum_{k=1}^n\sum_{p=1}^na_{ki}a_{pj}\delta_{kp}^q,\quad
\sum_{k=1}^n\xi_{ij}^ka_{qk}=\sum_{k=1}^n\sum_{p=1}^na_{ki}a_{pj}\xi_{kp}^q\nonumber.
\end{eqnarray}
\begin{eqnarray}
\sum_{k=1}^n\gamma_{ij}^kb_{qk}=\sum_{k=1}^n\sum_{p=1}^nb_{ki}a_{pj}\gamma_{kp}^q,\quad \sum_{k=1}^n\delta_{ij}^kb_{qk}=\sum_{k=1}^n\sum_{p=1}^nb_{ki}a_{pj}\delta_{kp}^q,\quad
\sum_{k=1}^n\xi_{ij}^kb_{qk}=\sum_{k=1}^n\sum_{p=1}^nb_{ki}a_{pj}\xi_{kp}^q\nonumber.
\end{eqnarray}
\subsection{Classification in two-dimensional of BiHom-associative trialgebras}
Note that $\mathcal{BT}as_n^m$ denote $m^{th}$ isomorphism class of BiHom-associative trialgebra in dimension $n.$

\begin{theorem}\label{the1}
 Any two-dimensional complex BiHom-associative trialgebras are either associative or isomorphic to one of the following pairwise non-isomorphic BiHom-associative trialgebras :
\end{theorem}

$\mathcal{BT}as_2^1$ :	
$\begin{array}{ll}
e_1\dashv e_2=e_1,\\
e_2\dashv e_1=e_1,
\end{array}\quad$
$\begin{array}{ll}
e_2\dashv e_2=e_1,\\
e_1\vdash e_2=e_1,
\end{array}\quad$
$\begin{array}{ll}
e_2\vdash e_1=e_1,\\
e_2\bot e_2=e_1,
\end{array}\quad$
$\begin{array}{ll}
\alpha(e_2)=e_1,\\
\beta(e_2)=e_1.
\end{array}$\\

$\mathcal{BT}as_2^2$ :	
$\begin{array}{ll}
e_1\dashv e_2=e_1,\\
e_2\dashv e_1=e_1,\\
\end{array}\quad$
$\begin{array}{ll}
e_2\dashv e_2=e_1,\\
e_1\vdash e_2=e_1,\\
\end{array}\quad$
$\begin{array}{ll}
e_2\vdash e_1=e_1,\\
e_2\vdash e_2=e_1,\\
\end{array}$
$\begin{array}{ll}
e_2\bot e_2=e_1,\\
\alpha(e_2)=\beta(e_2)=e_1.
\end{array}$\\

$\mathcal{BT}as_2^3$ :	
$\begin{array}{ll}
e_2\dashv e_2=e_1,\\
e_2\vdash e_2=e_1,
\end{array}\quad$
$\begin{array}{ll}
e_2\bot e_1=e_1,\\
e_2\bot e_2=e_1,\\
\end{array}\quad$
$\begin{array}{ll}
\alpha(e_1)=\beta(e_1)=e_1,\\
\alpha(e_2)=\beta(e_2)=e_1+e_2.
\end{array}$\\

$\mathcal{BT}as_2^4$ :	
$\begin{array}{ll}
e_2\dashv e_1=e_1,\\
e_2\dashv e_2=e_1+e_2,
\end{array}\,$
$\begin{array}{ll}
e_2\vdash e_1=e_1,\\
e_2\vdash e_2=e_1+e_2,
\end{array}\,$
$\begin{array}{ll}
e_2\bot e_1=e_1,\\
e_2\bot e_2=e_1+e_2,
\end{array}$
$\begin{array}{ll}
\alpha(e_1)=\beta(e_1)=e_1,\\
\alpha(e_2)=\beta(e_2)=e_1+e_2.
\end{array}$\\

$\mathcal{BT}as_2^5$ :	
$\begin{array}{ll}
e_1\dashv e_2=e_1,\\
e_2\dashv e_2=e_2,\\
\end{array}\quad$
$\begin{array}{ll}
e_1\vdash e_2=e_1,\\
e_2\vdash e_2=e_2,\\
\end{array}\,$
$\begin{array}{ll}
e_1\bot e_2=e_1,\\
e_2\bot e_2=e_2,
\end{array}$
$\begin{array}{ll}
\alpha(e_1)=\beta(e_1)=e_1,\\
\alpha(e_2)=e_2,\,\beta(e_2)=e_1+e_2.
\end{array}$\\

$\mathcal{BT}as_2^6$ :	
$\begin{array}{ll}
e_2\dashv e_2=e_1,\\
e_2\vdash e_2=e_1,
\end{array}\quad$
$\begin{array}{ll}
e_2\bot e_2=e_1,\\
\alpha(e_1)=\beta(e_1)=e_1,\\
\end{array}\,$
$\begin{array}{ll}
\alpha(e_2)=e_2,\\
\beta(e_2)=e_1+e_2.
\end{array}$

$\mathcal{BT}as_2^7$ :	
$\begin{array}{ll}
e_1\dashv e_2=e_1,\\
e_1\dashv e_2=e_1+e_2,
\end{array}\,$
$\begin{array}{ll}
e_1\vdash e_2=e_1,\\
e_2\vdash e_2=e_1+e_2,
\end{array}\,$
$\begin{array}{ll}
e_1\bot e_2=e_1,\\
e_2\bot e_2=e_1+e_2,
\end{array}\,$
$\begin{array}{ll}
\alpha(e_1)=e_1,\\
\alpha(e_2)=e_1+e_2,
\end{array}$
$\begin{array}{ll}
\beta(e_1)=e_1,\\
\beta(e_2)=e_1+e_2.
\end{array}$

\subsection{Classification in three-dimensional of BiHom-associative trialgebras}
\begin{theorem}\label{the2}
 Any three-dimensional complex BiHom-associative trialgebras are either associative or isomorphic to one of the following pairwise non-isomorphic BiHom-associative trialgebra :
\end{theorem}
$\mathcal{BT}as_3^1$ :	
$\begin{array}{ll}
e_1\dashv e_2=e_1,\\
e_2\dashv e_2=e_1,\\
e_2\dashv e_3=e_1,
\end{array}\quad$
$\begin{array}{ll}
e_2\vdash e_1=e_1,\\
e_2\vdash e_2=e_1,\\
e_3\vdash e_2=e_1,
\end{array}\,$
$\begin{array}{ll}
e_2\bot e_2=e_1,\\
e_2\bot e_3=e_1,\\
e_3\bot e_2=e_1,
\end{array}$
$\begin{array}{ll}
\alpha(e_3)=e_3,\\
\beta(e_2)=e_1.
\end{array}$

$\mathcal{BT}as_3^2$ :	
$\begin{array}{ll}
e_1\dashv e_2=e_1,\\
e_2\dashv e_1=e_1,\\
e_3\dashv e_2=e_1,
\end{array}\quad$
$\begin{array}{ll}
e_2\vdash e_2=e_1,\\
e_2\vdash e_3=e_1,\\
e_3\vdash e_3=e_1,
\end{array}\,$
$\begin{array}{ll}
e_2\bot e_1=e_1,\\
e_2\bot e_3=e_1,\\
e_3\bot e_2=e_1,
\end{array}$
$\begin{array}{ll}
\alpha(e_3)=e_3,\\
\beta(e_2)=e_1,\\
\beta(e_3)=e_3.
\end{array}$

$\mathcal{BT}as_3^{3}$ :	
$\begin{array}{ll}
e_2\dashv e_3=e_1,\\
e_3\dashv e_1=e_1,\\
e_3\dashv e_2=e_1,
\end{array}\quad$
$\begin{array}{ll}
e_1\vdash e_2=e_1\\
e_2\vdash e_3=e_1,\\
e_3\vdash e_2=e_1,
\end{array}\,$
$\begin{array}{ll}
e_1\bot e_1=e_1,\\
e_1\bot e_2=e_1,\\
e_2\bot e_3=e_1,
\end{array}$
$\begin{array}{ll}
e_3\bot e_1=e_1,\\
\alpha(e_2)=e_2,\\
\beta(e_3)=e_3.
\end{array}$

$\mathcal{BT}as_3^4$ :	
$\begin{array}{ll}
e_1\dashv e_2=e_1,\\
e_2\dashv e_1=e_3,\\
e_2\dashv e_2=e_1+e_3,\\
\end{array}\quad$
$\begin{array}{ll}
e_1\vdash e_2=e_1\\
e_2\vdash e_1=e_1\\
e_2\vdash e_2=e_1+e_3,\\
\end{array}\,$
$\begin{array}{ll}
e_2\vdash e_3=e_3,\\
e_2\bot e_2=e_1,\\
e_2\bot e_3=e_1,\\
\end{array}$
$\begin{array}{ll}
e_3\bot e_3=e_3,\\
\alpha(e_2)=e_1,\\
\beta(e_2)=e_1.
\end{array}$

$\mathcal{BT}as_3^5$ :	
$\begin{array}{ll}
e_1\dashv e_2=e_1,\\
e_2\dashv e_1=e_3,\\
e_2\dashv e_2=e_1,\\
\end{array}\quad$
$\begin{array}{ll}
e_2\dashv e_3=e_1,\\
e_2\vdash e_1=e_1\\
e_2\vdash e_3=e_1
\end{array}\,$
$\begin{array}{ll}
e_3\vdash e_2=e_1,\\
e_1\bot e_2=e_1,\\
e_2\bot e_2=e_1,
\end{array}$
$\begin{array}{ll}
e_3\bot e_2=e_1,\\
\alpha(e_3)=\beta(e_3)=e_3,\\
\beta(e_2)=e_1.
\end{array}$

$\mathcal{BT}as_3^6$ :	
$\begin{array}{ll}
e_1\dashv e_2=e_1,\\
e_2\dashv e_1=e_1,\\
e_2\dashv e_2=e_1,
\end{array}\quad$
$\begin{array}{ll}
e_2\vdash e_2=e_1,\\
e_2\vdash e_3=e_1,\\
e_3\vdash e_3=e_1,
\end{array}\,$
$\begin{array}{ll}
e_2\bot e_1=e_1,\\
e_2\bot e_2=e_1,\\
e_2\bot e_3=e_1,
\end{array}$
$\begin{array}{ll}
e_3\bot e_3=e_1,\\
\alpha(e_3)=e_3,\\
\beta(e_2)=e_1.
\end{array}$

$\mathcal{BT}as_3^7$ :	
$\begin{array}{ll}
e_1\dashv e_1=e_1,\\
e_1\dashv e_2=e_1,\\
e_2\dashv e_3=e_1,
\end{array}\quad$
$\begin{array}{ll}
e_3\dashv e_1=e_1,\\
e_1\vdash e_1=e_1\\
e_2\vdash e_3=e_1,
\end{array}\,$
$\begin{array}{ll}
e_3\vdash e_1=e_1,\\
e_1\bot e_2=e_1,\\
e_2\bot e_3=e_1,
\end{array}$
$\begin{array}{ll}
e_3\bot e_2=e_1,\\
\alpha(e_2)=e_2,\\
\beta(e_3)=e_3.
\end{array}$

$\mathcal{BT}as_3^8$ :	
$\begin{array}{ll}
e_1\dashv e_1=e_1,\\
e_2\dashv e_3=e_1,\\
e_3\dashv e_1=e_1,
\end{array}\quad$
$\begin{array}{ll}
e_3\dashv e_2=e_1,\\
e_1\vdash e_2=e_1\\
e_2\vdash e_3=e_1,
\end{array}\,$
$\begin{array}{ll}
e_3\vdash e_1=e_1,\\
e_2\bot e_3=e_1,\\
e_3\bot e_1=e_1,
\end{array}$
$\begin{array}{ll}
e_3\bot e_2=e_1,\\
\alpha(e_2)=e_2,\\
\beta(e_3)=e_3.
\end{array}$

$\mathcal{BT}as_3^{9}$ :	
$\begin{array}{ll}
e_1\dashv e_1=e_1,\\
e_1\dashv e_2=e_1,\\
e_3\dashv e_1=e_1,
\end{array}\quad$
$\begin{array}{ll}
e_2\vdash e_1=e_1\\
e_3\vdash e_1=e_1,\\
e_3\vdash e_2=e_1,
\end{array}\,$
$\begin{array}{ll}
e_1\bot e_1=e_1,\\
e_3\bot e_1=e_1,\\
e_3\bot e_2=e_1,
\end{array}$
$\begin{array}{ll}
\alpha(e_2)=e_2,\\
\beta(e_3)=e_3.
\end{array}$

$\mathcal{BT}as_3^{10}$ :	
$\begin{array}{ll}
e_2\dashv e_1=e_2,\\
e_2\dashv e_2=e_2,\\
e_3\dashv e_1=e_2,
\end{array}\quad$
$\begin{array}{ll}
e_3\dashv e_2=e_2,\\
e_2\vdash e_1=e_2\\
e_3\vdash e_1=e_2,
\end{array}\,$
$\begin{array}{ll}
e_3\vdash e_2=e_2,\\
e_2\bot e_2=e_2,\\
e_3\bot e_1=e_2,
\end{array}$
$\begin{array}{ll}
e_3\bot e_2=e_2,\\
\alpha(e_1)=e_1,\\
\beta(e_3)=e_3.
\end{array}$

$\mathcal{BT}as_3^{11}$ :	
$\begin{array}{ll}
e_2\dashv e_1=e_3,\\
e_2\dashv e_3=e_3,\\
e_3\dashv e_1=e_3,
\end{array}\quad$
$\begin{array}{ll}
e_2\vdash e_1=e_3\\
e_3\vdash e_1=e_3,\\
e_3\vdash e_3=e_3,
\end{array}\,$
$\begin{array}{ll}
e_2\bot e_1=e_3,\\
e_2\bot e_3=e_3,\\
e_3\bot e_1=e_3,
\end{array}$
$\begin{array}{ll}
e_3\bot e_3=e_3,\\
\alpha(e_1)=e_1,\\
\beta(e_3)=e_3.
\end{array}$

$\mathcal{BT}as_3^{12}$ :	
$\begin{array}{ll}
e_2\dashv e_1=e_3,\\
e_3\dashv e_1=e_3,\\
e_3\dashv e_3=e_3,
\end{array}\quad$
$\begin{array}{ll}
e_1\vdash e_2=e_3\\
e_2\vdash e_1=e_3,\\
e_3\vdash e_1=e_3,
\end{array}\,$
$\begin{array}{ll}
e_2\bot e_1=e_3,\\
e_2\bot e_3=e_3,\\
e_3\bot e_3=e_3,
\end{array}$
$\begin{array}{ll}
\alpha(e_1)=e_1,\\
\beta(e_3)=e_3.
\end{array}$

$\mathcal{BT}as_3^{13}$ :	
$\begin{array}{ll}
e_1\dashv e_2=e_3,\\
e_2\dashv e_1=e_3,\\
e_3\dashv e_1=e_3,
\end{array}\quad$
$\begin{array}{ll}
e_3\dashv e_3=e_3,\\
e_2\vdash e_1=e_3\\
e_2\vdash e_3=e_3,
\end{array}\,$
$\begin{array}{ll}
e_3\vdash e_1=e_3,\\
e_1\bot e_2=e_3,\\
e_2\bot e_3=e_3,
\end{array}$
$\begin{array}{ll}
e_3\bot e_3=e_3,\\
\alpha(e_1)=e_1,\\
\beta(e_3)=e_3.
\end{array}$

$\mathcal{BT}as_3^{14}$ :	
$\begin{array}{ll}
e_1\dashv e_2=e_3,\\
e_2\dashv e_2=e_3,\\
e_2\dashv e_3=e_3,
\end{array}\quad$
$\begin{array}{ll}
e_2\vdash e_1=e_3\\
e_2\vdash e_2=e_3,\\
e_3\vdash e_2=e_3,
\end{array}\,$
$\begin{array}{ll}
e_3\vdash e_3=e_3,\\
e_1\bot e_2=e_3,\\
e_2\bot e_2=e_3,
\end{array}$
$\begin{array}{ll}
e_3\bot e_2=e_3,\\
\alpha(e_1)=e_1,\\
\beta(e_1)=e_1.
\end{array}$

$\mathcal{BT}as_3^{15}$ :	
$\begin{array}{ll}
e_2\dashv e_1=e_3,\\
e_2\dashv e_2=e_3,\\
e_2\dashv e_3=e_3,
\end{array}\quad$
$\begin{array}{ll}
e_2\vdash e_2=e_3\\
e_2\vdash e_3=e_3,\\
e_3\vdash e_3=e_3,
\end{array}\,$
$\begin{array}{ll}
e_2\bot e_1=e_3,\\
e_3\bot e_2=e_3,\\
e_3\bot e_3=e_3,
\end{array}$
$\begin{array}{ll}
\alpha(e_1)=e_1,\\
\beta(e_1)=e_1.
\end{array}$

$\mathcal{BT}as_3^{16}$ :	
$\begin{array}{ll}
e_2\dashv e_1=e_1+e_3,\\
e_2\dashv e_2=e_1+e_3,\\
e_2\dashv e_3=e_1+e_3,
\end{array}\quad$
$\begin{array}{ll}
e_3\dashv e_2=e_3,\\
e_1\vdash e_2=e_1+e_3\\
e_2\vdash e_1=e_1+e_3,
\end{array}\,$
$\begin{array}{ll}
e_2\bot e_2=e_1+e_3,\\
e_2\bot e_3=e_3,\\
e_3\bot e_2=e_1,
\end{array}$
$\begin{array}{ll}
\alpha(e_2)=e_1,\\
\alpha(e_3)=e_2,\\
\beta(e_2)=e_1.
\end{array}$

$\mathcal{BT}as_3^{17}$ :	
$\begin{array}{ll}
e_1\dashv e_1=e_1+e_3,\\
e_2\dashv e_1=e_1+e_3,\\
e_2\dashv e_2=e_1+e_3,
\end{array}\quad$
$\begin{array}{ll}
e_2\dashv e_3=e_1+e_3,\\
e_2\vdash e_1=e_1+e_3\\
e_2\vdash e_2=e_1+e_3,
\end{array}\,$
$\begin{array}{ll}
e_3\vdash e_2=e_1+e_3,\\
e_2\bot e_2=e_1+e_3,\\
e_3\bot e_2=e_1+e_3,
\end{array}$
$\begin{array}{ll}
\alpha(e_2)=e_1,\\
\alpha(e_3)=e_2,\\
\beta(e_2)=e_1.
\end{array}$

$\mathcal{BT}as_3^{18}$ :	
$\begin{array}{ll}
e_2\dashv e_1=e_1+e_3,\\
e_2\dashv e_2=e_1+e_3,\\
e_2\dashv e_3=e_1+e_3,
\end{array}\quad$
$\begin{array}{ll}
e_3\dashv e_2=e_1+e_3,\\
e_1\vdash e_2=e_1+e_3\\
e_2\vdash e_1=e_1+e_3,
\end{array}\,$
$\begin{array}{ll}
e_3\vdash e_2=e_1+e_3,\\
e_2\bot e_2=e_1+e_3,\\
e_2\bot e_3=e_1+e_3,
\end{array}$
$\begin{array}{ll}
\alpha(e_2)=e_1,\\
\alpha(e_3)=e_2,\\
\beta(e_2)=e_1.
\end{array}$

$\mathcal{BT}as_3^{19}$ :	
$\begin{array}{ll}
e_2\dashv e_2=e_1,\\
e_2\dashv e_3=e_1+e_3,\\
e_3\dashv e_3=e_3,\\
\end{array}\quad$
$\begin{array}{ll}
e_1\vdash e_2=e_1\\
e_2\vdash e_1=e_1\\
e_3\vdash e_3=e_1+e_3,\\
\end{array}\,$
$\begin{array}{ll}
e_2\bot e_1=e_1+e_3,\\
e_2\bot e_3=e_1+e_3,\\
e_3\bot e_3=e_1+e_3,
\end{array}$
$\begin{array}{ll}
\alpha(e_2)=e_1,\\
\beta(e_2)=e_1.
\end{array}$

$\mathcal{BT}as_3^{20}$ :	
$\begin{array}{ll}
e_1\dashv e_2=e_1+e_3,\\
e_2\dashv e_2=e_1+e_3,\\
e_2\dashv e_3=e_1+e_3,
\end{array}\quad$
$\begin{array}{ll}
e_2\vdash e_2=e_1+e_3\\
e_2\vdash e_3=e_1+e_3,\\
e_3\vdash e_2=e_1+e_3,
\end{array}\,$
$\begin{array}{ll}
e_2\bot e_1=e_1,\\
e_2\bot e_2=e_1+e_3,\\
e_3\bot e_2=e_1+e_3,
\end{array}$
$\begin{array}{ll}
\alpha(e_2)=e_1,\\
\alpha(e_3)=e_2,\\
\beta(e_2)=e_1.
\end{array}$

$\mathcal{BT}as_3^{21}$ :	
$\begin{array}{ll}
e_1\dashv e_2=e_1,\\
e_2\dashv e_1=e_1+e_3,\\
e_2\dashv e_2=e_1+e_3,\\
\end{array}\quad$
$\begin{array}{ll}
e_1\vdash e_2=e_1+e_3\\
e_2\vdash e_1=e_1+e_3\\
e_2\vdash e_2=e_1+e_3,\\
\end{array}\,$
$\begin{array}{ll}
e_2\bot e_1=e_1,\\
e_2\bot e_3=e_1+e_3,\\
e_3\bot e_3=e_1+e_3,\\
\end{array}$
$\begin{array}{ll}
\alpha(e_2)=e_1,\\
\beta(e_2)=e_1.
\end{array}$

$\mathcal{BT}as_3^{22}$ :	
$\begin{array}{ll}
e_1\dashv e_2=e_3,\\
e_2\dashv e_2=e_1+e_3,\\
e_3\dashv e_3=e_1+e_3,
\end{array}\quad$
$\begin{array}{ll}
e_2\vdash e_1=e_1\\
e_2\vdash e_2=e_1+e_3,\\
e_3\vdash e_2=e_1+e_3,
\end{array}\,$
$\begin{array}{ll}
e_2\bot e_2=e_1+e_3,\\
e_2\bot e_3=e_1+e_3,\\
e_3\bot e_2=e_1,
\end{array}$
$\begin{array}{ll}
\alpha(e_2)=e_1,\\
\alpha(e_3)=e_2,\\
\beta(e_2)=e_1.
\end{array}$

$\mathcal{BT}as_3^{23}$ :	
$\begin{array}{ll}
e_2\dashv e_1=e_3,\\
e_2\dashv e_2=e_1+e_3,\\
e_2\dashv e_3=e_1+e_3,
\end{array}\quad$
$\begin{array}{ll}
e_3\dashv e_2=e_1+e_3,\\
e_2\vdash e_1=e_3\\
e_2\vdash e_2=e_1,
\end{array}\,$
$\begin{array}{ll}
e_2\bot e_1=e_1+e_3,\\
e_2\bot e_2=e_3,\\
e_2\bot e_3=e_3,
\end{array}$
$\begin{array}{ll}
\alpha(e_2)=e_1,\\
\alpha(e_3)=e_2,\\
\beta(e_2)=e_1.
\end{array}$

$\mathcal{BT}as_3^{24}$ :	
$\begin{array}{ll}
e_2\dashv e_1=e_3,\\
e_2\dashv e_2=e_3,\\
e_2\dashv e_3=e_1+e_3,
\end{array}\quad$
$\begin{array}{ll}
e_2\vdash e_2=e_1+e_3\\
e_2\vdash e_3=e_3,\\
e_2\vdash e_3=e_1,
\end{array}\,$
$\begin{array}{ll}
e_1\bot e_2=e_1+e_3,\\
e_2\bot e_3=e_1+e_3,\\
e_3\bot e_2=e_3,\quad
\end{array}$
$\begin{array}{ll}
\alpha(e_2)=e_1,\\
\alpha(e_3)=e_2,\\
\beta(e_2)=e_1.
\end{array}$

\section{Derivations of low-dimensional BiHom-associative trialgebras}
\begin{definition}\label{dia2}
An $\alpha\beta$-derivation of the BiHom-associative trialgebras $\mathcal{A}$ is a linear transformation : $\mathcal{D} : \mathcal{A} \rightarrow \mathcal{A}$ satisfying
\begin{eqnarray}
\alpha\circ d=d\circ\alpha&,& \beta\circ d=d\circ\beta\\
d(x\dashv y)&=&d(x)\dashv\alpha\beta(y)+\alpha\beta(x)\dashv d(y)\\
d(x\vdash y)&=&d(x)\vdash\alpha\beta(y)+\alpha\beta(x)\vdash d(y)\\
d(x\bot y)&=&d(x)\bot\alpha\beta(y)+\alpha\beta(x)\bot d(y),
\end{eqnarray}
for all $x, y\in \mathcal{A}.$
 \end{definition}

This section illustrates in depth, $\alpha\beta$-derivation of BiHom-associative trialgebras in dimension two and three over the field $\mathbb{F}.$ Let
$\left\{e_1,e_2, e_3,\cdots, e_n\right\}$ be a basis of an n-dimensional BiHom-associative trialgebra $\mathcal{A}.$ The product of basis is denoted by
\begin{eqnarray}
d(e_p)=\sum_{q=1}^nd_{qp}e_q\nonumber.
\end{eqnarray}

We have
\begin{eqnarray}
\sum_{p=1}^nd_{pk}a_{qp}=\sum_{p=1}^na_{pk}d_{qp}& ; &\sum_{p=1}^nd_{pk}b_{qp}=\sum_{p=1}^nb_{pk}d_{qp}\label{deq1},\\
\sum_{k=1}^n\gamma_{ij}^pd_{rp}=\sum_{k=1}^n\sum_{q=1}^n\sum_{p=1}d_{ki}b_{pj}a_{qp}\gamma_{kq}^r&+&\sum_{q=1}^n\sum_{k=1}^n\sum_{p=1}b_{ki}a_{qk}d_{pj}\gamma_{qp}^r\label{deq2},\\
\sum_{k=1}^n\delta_{ij}^pd_{rp}=\sum_{k=1}^n\sum_{q=1}^n\sum_{p=1}d_{ki}b_{pj}a_{qp}\delta_{kq}^r&+&\sum_{q=1}^n\sum_{k=1}^n\sum_{p=1}b_{ki}a_{qk}d_{pj}\delta_{qp}^r\label{deq3},\\
\sum_{k=1}^n\xi_{ij}^pd_{rp}=\sum_{k=1}^n\sum_{q=1}^n\sum_{p=1}d_{ki}b_{pj}a_{qp}\xi_{kq}^r&+&\sum_{q=1}^n\sum_{k=1}^n\sum_{p=1}b_{ki}a_{qk}d_{pj}\xi_{qp}^r\label{deq4}.
\end{eqnarray}

\begin{theorem}
The derivation of two-dimensional BiHom-associative trialgebras has the following form :
\end{theorem}

\begin{tabular}{||c||c||c||c||c||c||c||c||c||c||c||c||}
\hline
IC&Der$(\mathcal{D})$ &$Dim(\mathcal{D})$&IC&Der$(\mathcal{D})$&$Dim(\mathcal{D})$\\
			\hline
$\mathcal{BT}as_2^1$&
$\left(\begin{array}{cccc}
0&0\\
d_{21}&0
\end{array}
\right)$
&
1
&
$\mathcal{BT}as_2^{2}$&
$\left(\begin{array}{cccc}
0&0\\
d_{21}&0
\end{array}
\right)$
&
1
\\ \hline
$\mathcal{BT}as_2^{6}$&
$\left(\begin{array}{cccc}
0&0\\
d_{21}&0
\end{array}
\right)$
&
1
&
&
&
\\ \hline
\end{tabular}

\begin{theorem}\label{dtheo2}
The derivation of three-dimensional BiHom-associative trialgebras has the following form :
\end{theorem}

\begin{tabular}{||c||c||c||c||c||c||c||c||c||c||c||c||}
\hline
IC&Der$(\mathcal{D})$ &$Dim(\mathcal{D})$&IC&Der$(\mathcal{D})$&$Dim(\mathcal{D})$\\
			\hline
$\mathcal{BT}as_3^1$&
$\left(\begin{array}{cccc}
0&0&0\\
d_{21}&0&0\\
0&0&d_{33}
\end{array}
\right)$
&
2
&
$\mathcal{BT}as_3^{2}$&
$\left(\begin{array}{cccc}
0&0&0\\
d_{21}&0&0\\
0&0&0
\end{array}
\right)$
&
1
\\ \hline
$\mathcal{BT}as_3^3$&
$\left(\begin{array}{cccc}
0&0&0\\
0&d_{22}&0\\
0&0&d_{33}
\end{array}
\right)$
&
2
&
$\mathcal{BT}as_3^{4}$&
$\left(\begin{array}{cccc}
0&0&0\\
d_{21}&0&d_{23}\\
0&0&0
\end{array}
\right)$
&
2
\\ \hline
$\mathcal{BT}as_3^5$&
$\left(\begin{array}{cccc}
0&0&0\\
d_{21}&0&0\\
0&0&0
\end{array}
\right)$
&
1
&
$\mathcal{BT}as_3^{6}$&
$\left(\begin{array}{cccc}
0&0&0\\
d_{21}&0&0\\
0&0&d_{33}
\end{array}
\right)$
&
2
\\ \hline
$\mathcal{BT}as_3^7$&
$\left(\begin{array}{cccc}
0&0&0\\
0&d_{22}&0\\
0&0&d_{33}
\end{array}
\right)$
&
2
&
$\mathcal{BT}as_3^{8}$&
$\left(\begin{array}{cccc}
0&0&0\\
0&d_{22}&0\\
0&0&d_{33}
\end{array}
\right)$
&
2
\\ \hline
$\mathcal{BT}as_3^9$&
$\left(\begin{array}{cccc}
0&0&0\\
0&d_{22}&0\\
0&0&d_{33}
\end{array}
\right)$
&
2
&
$\mathcal{BT}as_3^{10}$&
$\left(\begin{array}{cccc}
d_{11}&0&0\\
0&0&0\\
0&0&d_{33}
\end{array}
\right)$
&
2
\\ \hline
$\mathcal{BT}as_3^{11}$&
$\left(\begin{array}{cccc}
d_{11}&0&0\\
0&d_{22}&0\\
0&0&0
\end{array}
\right)$
&
2
&
$\mathcal{BT}as_3^{12}$&
$\left(\begin{array}{cccc}
d_{11}&0&0\\
0&d_{22}&0\\
0&0&0
\end{array}
\right)$
&
2
\\ \hline
$\mathcal{BT}as_3^{13}$&
$\left(\begin{array}{cccc}
d_{11}&0&0\\
0&d_{22}&0\\
0&0&0
\end{array}
\right)$
&
2
&
$\mathcal{BT}as_3^{14}$&
$\left(\begin{array}{cccc}
d_{11}&0&0\\
0&d_{22}&d_{23}\\
0&0&0
\end{array}
\right)$
&
3
\\ \hline
$\mathcal{BT}as_3^{15}$&
$\left(\begin{array}{cccc}
d_{11}&0&0\\
0&0&d_{23}\\
0&0&0
\end{array}
\right)$
&
2
&
$\mathcal{BT}as_3^{19}$&
$\left(\begin{array}{cccc}
0&0&0\\
d_{21}&0&d_{23}\\
0&0&0
\end{array}
\right)$
&
2
\\ \hline
\end{tabular}

\begin{proof}
Departing from Theorem \ref{dtheo2}, we provide the proof only for one case in order to illustrate the used approach. The other cases can be handled similarly with or no
modification(s). We refer to Theorem  \ref{dtheo2} and use the system of equations  (\ref{deq1}), (\ref{deq2}), (\ref{deq3}), (\ref{deq4}) of operations $\dashv, \vdash$ and $\bot$,
respectively, to obtain all derivations of BiHom-associative trialgebra in dimension three. Let's consider ${Trias}_3^{1}$. Applying the systems of equations (\ref{deq1}), (\ref{deq2}), (\ref{deq3}),
(\ref{deq4}), we therefore get the derivation for ${Trias}_3^{1}$ as follows
$d_1=\left(\begin{array}{ccc}
0&0&0\\
d_{21}&0&0\\
0&0&d_{33}
\end{array}
\right)$.
Clary, $d_{1}(e_1)=e_2$, $d_{2}(e_3)=e_3$. Hence, the derivations of ${Trias}_3^{1}$ are indicated as follows\\
$d_1=\left(\begin{array}{ccc}
0&0&0\\
1&0&0\\
0&0&0
\end{array}
\right)$,\,$d_2=\left(\begin{array}{ccc}
0&0&0\\
0&0&0\\
0&0&1
\end{array}
\right)$ is basis of $Der(\mathcal{D})$ and Dim$Der(\mathcal{D})=2.$ The derivations of the remaining parts of three-dimension  associative trialgebras can be
handled in a similar manner as depicted above.
\end{proof}

\begin{corollary}\,
\begin{itemize}
	\item The dimensions of the derivations of two-dimensional BiHom-associative trialgebras range between zero and two.
	\item The dimensions of the derivations of three-dimensional BiHom-associative trialgebras range between zero and three.
\end{itemize}
\end{corollary}

\section{Centroids of low-dimensional BiHom-associative trialgebras}
\subsection{Properties of centroids of BiHom-associative trialgebras}
In this section, we display the following results as far as properties of centroids of BiHom-associative trialgebras $\mathcal{A}$ are concerned.
\begin{definition}
Let $\mathcal{H}$ be a nonempty subset of $\mathcal{A}$. The subset
\begin{equation}
Z_{\mathcal{A}}(\mathcal{H})=\left\{x\in\mathcal{H}\, |\, \alpha\beta(x)\bullet \mathcal{H} = \mathcal{H}\bullet\alpha\beta(x)=0\right\},
\end{equation}
is said to be centralized of $\mathcal{H}$ in $\mathcal{A}$, where the $\bullet$ is $\dashv, \vdash$ and $\bot,$ respectively.
\end{definition}

\begin{definition}
Let $\psi\in End(\mathcal{A})$. If $\psi(\mathcal{A})\subseteq Z(\mathcal{A})$ and $\psi(\mathcal{A}^2)=0$, then $\psi$ is called a central $\alpha\beta$-derivation.
The set of all central $\alpha\beta$-derivations of $\mathcal{A}$ is  denoted by $\mathcal{C}(\mathcal{A})$.
\end{definition}

\begin{proposition}
Consider $(\mathcal{A}, \dashv, \vdash, \bot, \alpha, \beta)$ a BiHom-associative trialgebra. Then,
\begin{enumerate}
	\item [i)]$\Gamma(\mathcal{A})Der(\mathcal{A})\subseteq Der(\mathcal{A})$.
		\item [ii)]$\left[\Gamma(\mathcal{A}), Dr(\mathcal{A})\right]\subseteq\Gamma(\mathcal{A}).$
	\item [iii)]$\left[\Gamma(\mathcal{A}), \Gamma(\mathcal{A})\right](\mathcal{A})\subseteq \Gamma(\mathcal{A})$ and $\left[\Gamma(\mathcal{A}), \Gamma(\mathcal{A})\right](\mathcal{A}^2)=0.$
\end{enumerate}
 \end{proposition}
\begin{proof}
The proof
 of parts $i)-iii)$ is straightforward relying on definitions of $\alpha\beta$-derivation and centroids.
\end{proof}
Now, we shall introduce centroid for BiHom-associative trialgebras.
\begin{definition}
Let $(\mathcal{A}, \dashv, \vdash, \bot, \alpha, \beta)$ be a  BiHom-associative trialgebra. A linear map
 $\psi : \mathcal{A}\rightarrow \mathcal{A}$ is called an element of $(\alpha, \beta)$-element of centroids on $\mathcal{A}$ if, for all $x, y\in \mathcal{A}$,
\begin{eqnarray}
\alpha\circ\psi&=&\psi\circ\alpha,\quad \beta\circ\psi=\psi\circ\beta,\\
\psi(x)\dashv \alpha\beta(y)&=&\psi(x)\dashv\psi(y)=\alpha\beta(x)\dashv \psi(y),\\
 \psi(x)\vdash \alpha\psi(y)&=&\psi(x)\vdash \psi(y)=\alpha\beta(x)\vdash \psi(y),\\
 \psi(x)\bot \alpha\beta(y)&=&\psi(x)\bot\psi(y)=\alpha\beta(x)\bot \psi(y).
\end{eqnarray}
The set of all  $(\alpha,\beta)$-elements of centroid of $\mathcal{A}$ is denoted $Cent_{(\alpha, \beta)}(\mathcal{A})$.
The centroid of $\mathcal{A}$ is denoted $Cent(\mathcal{A})$.
 \end{definition}

\begin{proposition}
Let $(\mathcal{A}, \dashv, \vdash, \bot, \alpha, \beta)$ be a BiHom-associative trialgebra and $\varphi\in Cent(\mathcal{A}),\, d\in Der(\mathcal{A}).$
Then, $\varphi\circ d$ is a $\alpha\beta$-derivation of $\mathcal{A}.$
\end{proposition}
\begin{proof}
Indeed, if $x, y\in \mathcal{A}$, then
$$\begin{array}{ll}
(\varphi\circ d)(x\bullet y)
&= \varphi(d(x)\bullet\alpha\beta(y)+\alpha\beta(x)\bullet d(y))\\
&= \varphi(d(x)\bullet y)+\varphi(x\bullet d(y))=(\varphi\circ d)(x)\bullet\alpha\beta(y)+\alpha\beta(x)\bullet(\varphi\circ d)(y),
\end{array}$$
where $\bullet$ is $\dashv, \vdash$ and $\bot$, respectively.
\end{proof}

\begin{proposition}
Let $(\mathcal{A}, \dashv, \vdash, \bot, \alpha, \beta)$ be a BiHom-associative trialgebra over a field $\mathbb{F}$. Hence, $\mathcal{C}(\mathcal{A})=Cent(\mathcal{A})\cap Der(\mathcal{A}).$
\end{proposition}

\begin{proof}
If $\psi\in Cent(\mathcal{A})\cap Der(\mathcal{A})$, then by definition of $Cent(\mathcal{A}$ and $Der(\mathcal{A})$, we have

$\psi(x\bullet y)=\psi(x)\bullet\alpha\beta(y)+\alpha\beta(x)\bullet\psi(y)$ and $\psi(x\bullet y)=\psi(x)\circ\alpha\beta(y)=\alpha\beta(x)\circ\psi(y)$ for $x,y\in \mathcal{A}.$
The yields $\psi(\mathcal{A}\mathcal{A})=0$ and $\psi(\mathcal{A})\subseteq  Z(\mathcal{A})$ i.e $Cent(\mathcal{A})\cap Der(\mathcal{A})\subseteq Cent(\mathcal{A}).$
The inverse is obvious since $\mathcal{C}(\mathcal{A})$ is in both $Cent(\mathcal{A})$ and $Der(\mathcal{A}),$ where $\bullet$ is $\dashv, \vdash$ and $\bot$, respectively.
\end{proof}

\begin{proposition}
Let $(\mathcal{A}, \dashv, \vdash, \bot, \alpha, \beta)$ be a BiHom-associative trialgebra. Therefore,  for any $d\in Der(\mathcal{A})$ and $\varphi\in Cent(\mathcal{A})$, we have
\begin{enumerate}
	\item [(i)] The composition $d\circ\varphi$ is in $Cent(\mathcal{A})$, if and only if $\varphi\circ d$ is a central $\alpha\beta$-derivation of $\mathcal{A}.$
		\item [(ii)] The  composition $d\circ\varphi$ is a $\alpha\beta$-derivation of $\mathcal{A}$, if and only if $\left[d,\varphi\right]$ is a central $\alpha\beta$-derivation of $\mathcal{A}.$
\end{enumerate}
\end{proposition}

\begin{proof}
\begin{enumerate}
	\item [i)]For any $\varphi\in Cent(\mathcal{A}),\, d\in Der(\mathcal{A}),\, \forall\,x,y\in \mathcal{A}$, we have
	$$\begin{array}{ll}
d\circ\varphi(x\bullet y)=d\circ\varphi(x)\bullet y
&=d\circ\varphi(x)\bullet y+\varphi(x)\bullet d(y)\\
&=d\circ\varphi(x)\bullet y+\varphi\circ d(x\bullet y)-\varphi\circ d(x)\bullet y.
\end{array}$$
Thus, $(d\circ\varphi-\varphi\circ d)(x\bullet y)=(d\bullet\varphi-\varphi\circ d)(x)\bullet y.$
	\item [ii)] Let $d\circ\varphi\in Der(\mathcal{A})$. Using $\left[d,\varphi\right]\in Cent(\mathcal{A})$, we get
	\begin{equation}\label{eq1}
	\left[d,\varphi\right](x\bullet y)=(\left[d, \varphi\right](x))\bullet\alpha\beta(y)=\alpha\beta(x)\bullet(\left[d,\varphi\right](y))
	\end{equation}
	On the other side, $\left[d, \varphi\right]d\circ\varphi-\varphi\circ d$ and $d\circ\varphi, \varphi\circ d\in Der(\mathcal{A}).$
	Therefore,	
	\begin{equation}\label{eq2}
\left[d, \varphi\right](x\bullet y)=(d(\varphi\circ(x))\bullet\alpha\beta(y)+\alpha\beta(x)\bullet(d\circ\varphi(y))-(\varphi\circ d(x))\bullet\alpha\beta(y)-\alpha\beta(x)\bullet(\varphi\circ d(y)).
\end{equation}
Referring to (\ref{eq1}) and (\ref{eq2}), we get $\alpha\beta(x)\bullet(\left[d, \varphi\right])(y)=(\left[d, \varphi\right])(x)\bullet\alpha\beta(y)=0.$

\noindent At this stage of analysis,let $\left[d, \varphi\right]$ be a central $\alpha\beta$-derivation of $\mathcal{A}$. Then,
$$\begin{array}{ll}
d\circ\varphi(x\bullet y)
&=\left[d\circ\varphi\right](x\bullet y)+(\varphi\circ d)(x\bullet y)\\
&=\varphi(\circ d(x)\bullet\alpha\beta(y))+\varphi(\alpha\beta(x)\bullet d(y))\\
&=(\varphi\circ d)(x)\bullet\alpha\beta(y)+\alpha\beta(x)\bullet(\varphi\circ d)(y),
\end{array}$$
\end{enumerate}
where $\bullet$ represents the products $\dashv, \vdash$ and $\bot$, respectively.
\end{proof}
\subsection{Centroids of BiHom-associative trialgebras}
This section adresses in details the centroid of BiHom-associative trialgebras in dimension two and three over the field $\mathbb{F}.$ Let
$\left\{e_1,e_2, e_3,\cdots, e_n\right\}$ be a basis of an n-dimensional BiHom-associative trialgebra $\mathcal{A}.$ The product of basis is determined by
\begin{eqnarray}
\psi(e_i)=\sum_{j=1}^nc_{ji}e_j \nonumber.
\end{eqnarray}

We,therefore,have
\begin{eqnarray}
\sum_{p=1}^nc_{pi}a_{qp}=\sum_{p=1}^na_{pi}c_{qp}& ; & \sum_{p=1}^nc_{pi}b_{qp}=\sum_{p=1}^nb_{pi}c_{qp},\\
\sum_{k=1}^n\sum_{q=1}\sum_{p=1}c_{ki}b_{pj}a_{qp}\gamma_{kq}^r=\sum_{p=1}^n\sum_{q=1}^nc_{pi}c_{qj}\gamma_{pq}^r&=&\sum_{q=1}^n\sum_{k=1}^n\sum_{p=1}^nb_{ki}a_{qk}c_{pj}\gamma_{qp}^r,\\
\sum_{k=1}^n\sum_{q=1}\sum_{p=1}c_{ki}b_{pj}a_{qp}\delta_{kq}^r=\sum_{p=1}^n\sum_{q=1}^nc_{pi}c_{qj}\delta_{pq}^r&=&\sum_{q=1}^n\sum_{k=1}^n\sum_{p=1}^nb_{ki}a_{qk}c_{pj}\delta_{qp}^r,\\
\sum_{k=1}^n\sum_{q=1}\sum_{p=1}c_{ki}b_{pj}a_{qp}\xi_{kq}^r=\sum_{p=1}^n\sum_{q=1}^nc_{pi}c_{qj}\xi_{pq}^r&=&\sum_{q=1}^n\sum_{k=1}^n\sum_{p=1}^nb_{ki}a_{qk}c_{pj}\xi_{qp}^r.
\end{eqnarray}


\begin{theorem}
The centroids of three-dimensional complex BiHom-associative trialgebras are\\ illustrated as follows :
\end{theorem}

\begin{tabular}{||c||c||c||c||c||c||c||c||c||c||c||c||}
\hline
IC&$Cent(\mathcal{A})$ &$Dim(Cent(\mathcal{A}))$&IC&$Cent(\mathcal{A})$&$Dim(Cent(\mathcal{A}))$\\
			\hline
$\mathcal{BT}as_3^1$&
$\left(\begin{array}{cccc}
0&0&0\\
c_{21}&0&0\\
0&0&0
\end{array}
\right)$
&
1
&
$\mathcal{BT}as_3^{2}$&
$\left(\begin{array}{cccc}
0&0&0\\
c_{21}&0&0\\
0&0&0
\end{array}
\right)$
&
1
\\ \hline
$\mathcal{BT}as_3^3$&
$\left(\begin{array}{cccc}
0&0&0\\
0&0&0\\
0&0&c_{33}
\end{array}
\right)$
&
1
&
$\mathcal{BT}as_3^{4}$&
$\left(\begin{array}{cccc}
0&0&0\\
0&0&0\\
c_{31}&0&0
\end{array}
\right)$
&
1
\\ \hline
$\mathcal{BT}as_3^5$&
$\left(\begin{array}{cccc}
0&0&0\\
c_{21}&0&0\\
0&0&0
\end{array}
\right)$
&
1
&
$\mathcal{BT}as_3^{6}$&
$\left(\begin{array}{cccc}
0&0&0\\
c_{21}&0&0\\
0&0&0
\end{array}
\right)$
&
1
\\ \hline
$\mathcal{BT}as_3^7$&
$\left(\begin{array}{cccc}
0&0&0\\
0&0&0\\
0&0&c_{33}
\end{array}
\right)$
&
1
&
$\mathcal{BT}as_3^{8}$&
$\left(\begin{array}{cccc}
0&0&0\\
0&0&0\\
0&0&c_{33}
\end{array}
\right)$
&
1
\\ \hline
$\mathcal{BT}as_3^9$&
$\left(\begin{array}{cccc}
0&0&0\\
0&0&0\\
0&0&c_{33}
\end{array}
\right)$
&
1
&
$\mathcal{BT}as_3^{10}$&
$\left(\begin{array}{cccc}
0&0&0\\
0&0&0\\
0&0&c_{33}
\end{array}
\right)$
&
1
\\ \hline
$\mathcal{BT}as_3^{11}$&
$\left(\begin{array}{cccc}
0&0&0\\
0&c_{22}&0\\
0&0&0
\end{array}
\right)$
&
1
&
$\mathcal{BT}as_3^{12}$&
$\left(\begin{array}{cccc}
0&0&0\\
0&c_{22}&0\\
0&0&0
\end{array}
\right)$
&
1
\\ \hline
$\mathcal{BT}as_3^{13}$&
$\left(\begin{array}{cccc}
0&0&0\\
0&c_{22}&0\\
0&0&0
\end{array}
\right)$
&
1
&
$\mathcal{BT}as_3^{15}$&
$\left(\begin{array}{cccc}
c_{11}&0&0\\
0&0&0\\
0&0&0
\end{array}
\right)$
&
1
\\ \hline
$\mathcal{BT}as_3^{16}$&
$\left(\begin{array}{cccc}
0&0&0\\
0&0&0\\
c_{31}&0&0
\end{array}
\right)$
&
1
&
$\mathcal{BT}as_3^{18}$&
$\left(\begin{array}{cccc}
0&0&0\\
0&0&0\\
c_{31}&0&0
\end{array}
\right)$
&
1
\\ \hline
\end{tabular}

\begin{tabular}{||c||c||c||c||c||c||c||c||c||c||c||c||}
\hline
IC&$Cent(\mathcal{A})$ &$Dim(Cent(\mathcal{A}))$&IC&$Cent(\mathcal{A})$&$Dim(Cent(\mathcal{A}))$\\
			\hline
$\mathcal{BT}as_3^{19}$&
$\left(\begin{array}{cccc}
0&0&0\\
0&0&0\\
c_{13}&0&0
\end{array}
\right)$
&
1
&
$\mathcal{BT}as_3^{20}$&
$\left(\begin{array}{cccc}
0&0&0\\
0&0&0\\
c_{31}&0&0
\end{array}
\right)$
&
1
\\ \hline			
$\mathcal{BT}as_3^{21}$&
$\left(\begin{array}{cccc}
0&0&0\\
0&0&0\\
c_{31}&0&0
\end{array}
\right)$
&
1
&
$\mathcal{BT}as_3^{22}$&
$\left(\begin{array}{cccc}
0&0&0\\
0&0&0\\
c_{31}&0&0
\end{array}
\right)$
&
1
\\ \hline			
$\mathcal{BT}as_3^{23}$&
$\left(\begin{array}{cccc}
0&0&0\\
0&0&0\\
c_{31}&0&0
\end{array}
\right)$
&
1
&
$\mathcal{BT}as_3^{24}$&
$\left(\begin{array}{cccc}
0&0&0\\
0&0&0\\
c_{31}&0&0
\end{array}
\right)$
&
1.
\\ \hline
\end{tabular}

\begin{corollary}\,
\begin{itemize}
	\item The dimensions of the centroids of two-dimensional BiHom-associative trialgebras are zero.
	\item The dimensions of the centroids of three-dimensional BiHom-associative trialgebras range between zero and one.
\end{itemize}
\end{corollary}

\end{document}